\documentclass{amsart}

\newtheorem{theorem}{Theorem}[section]

\newtheorem{example}{Example}[section]

\newtheorem{definition}{Definition}[section]

\newtheorem{corollary}{Corollary}[section]

\numberwithin{equation}{section}

\newcommand{\abs}[1]{\left | #1\right |}

\begin{document}

\title[slowly oscillating double sequences]{Functions preserving slowly oscillating double sequences}
\author[ Huseyin Cakalli \& Richard F. Patterson ]{\;\;\;  Huseyin Cakalli* \& Richard F. Patterson**\\
*D\lowercase{epartment of} M\lowercase{athematics}, M\lowercase{altepe} U\lowercase{niversity}, M\lowercase{altepe}-I\lowercase{stanbul}, Turkey\\
**D\lowercase{epartment of} M\lowercase{athematics and} S\lowercase{tatistics},
U\lowercase{niversity of} N\lowercase{orth} F\lowercase{lorida}, F\lowercase{lorida}, USA}
\address{Department of Mathematics, Maltepe University, Maltepe- Istanbul, Turkey}
\email{huseyincakalli@maltepe.edu.tr; hcakalli@gmail.com}
\address{Department of Mathematics and Statistics,
University of North Florida Jacksonville, Florida, 32224}
\email{rpatters@unf.edu}

\date{\today}


\maketitle
\begin{abstract}
A double sequence $\textbf{x}=\{x_{k,l}\}$ of points in $\textbf{R}$ is slowly oscillating if for any given $\varepsilon>0$, there exist $\alpha=\alpha(\varepsilon)>0$, $\delta=\delta (\varepsilon) >0$, and $N=N(\varepsilon)$ such that $|x_{k,l}-x_{s,t}|<\varepsilon$ whenever $k,l\geq N(\varepsilon)$ and $k\leq s \leq (1+\alpha)k$, $l\leq t \leq (1+\delta)l$. We study continuity type properties of factorable double functions defined on a double subset $A\times A$ of $\textbf{R}^{2}$ into $\textbf{R}$, and obtain interesting results related to uniform continuity, sequential continuity, and a newly introduced type of continuity of factorable double functions defined on a double subset $A\times A$ of $\textbf{R}^{2}$ into $\textbf{R}$.

\end{abstract}

\section{\bf Introduction}
In 1900, Pringsheim (\cite{Pring1}) introduced the concept of convergence of real double sequences. Four years later, Hardy (\cite{HardyOntheconvergenceofcertainmultipleseries}) introduced the notion of regular convergence for double sequences in the sense that double sequence has a limit in Pringsheim's sense and has one sided limits (see also \cite{ Ham1, robison}).
A considerable number of papers which appeared in recent years study double sequences from various points of view (see \cite{AlotaibiandursaleenandAlghamdi,DjurcicandKocinacandZizovic,DuttaACharacterizationoftheClassofStatisticallyPre-CauchyDoubleSequencesofFuzzyNumbers,MursaleenandMohiuddineandSyed,PattersonAtheoremonentirefourdimensionalsummabilitymethods,PattersonFourdimensionalmatrixcharacterizationP-convergencefieldsofsummabilitymethods, PattersonRH-regularTransformationsWhichSumsaGivenDoubleSequence,PattersonandSavasAsymptoticEquivalenceofDoubleSequences}). Some results in the investigation are generalizations of known results concerning simple sequences to certain classes of double sequences, while other results reflect a specific nature of the Pringsheim convergence (e.g., the fact that a double sequence may converge without being bounded).
First usage of the slowly oscillating concept of real single sequences goes back to beginning of twentieth century (\cite[1907, Hardy]{HardySlowlysingletitleSometheoremsconcerninginfiniteseries}, and   (\cite[1910, Landou]{LandauSlowlyoscillatingtitleÜberdieBedeutungeinigerneuerGrenzwertsätzederHerrenHardyundAxel}) while the slowly oscillating concept of real double  sequences seems to be first studied in  \cite[1939, Knopp]{KnoppLimitierungsUmkehrsatzefurDoppelfolgen} (see also \cite{HardyTheoremsrelatingtothesummabilityandconvergenceofslowlyoscillating}, and \cite{MoriczTauberiantheoremsforCesarosummabledoublesequences}).

The aim of this paper is to investigate slowly oscillating double sequences and newly defined types of continuities for factorable double functions.

\section{\bf Preliminaries}
\begin{definition} [Pringsheim, 1900 \cite{Pring1}] A double sequence ${\bf x}=\{x_{k,l}\}$ is Cauchy provided that, given an $\epsilon > 0$ there exists an $N \in {\bf N}$ such that $|x_{k,l} - x_{s,t}| < \epsilon$ whenever $k,l,s,t > N$.
\end {definition}

\begin{definition}
[Pringsheim, 1900 \cite{Pring1}] A double sequence ${\bf x}=\{x_{k,l}\}$
has a {\bf Pringsheim limit} $L$ (denoted by P-$\lim x=L$) provided that, given an $\epsilon > 0$ there
exists an $N \in {\bf N}$ such that $\left|x_{k,l} - L\right| < \epsilon$
whenever $k,l > N$.  Such an $\bf x$ is described more briefly as ``$P$-convergent''.
\end{definition}
If P-$\lim |\textbf{x}| = \infty$, (equivalently, for every $\varepsilon > 0$  there are $n_{1}, n_{2}\in{N}$ such that $| x_{m,n} |>M$  whenever $m>n_{1}$, $n>n_{2}$), then $\textbf{x}=\{x_{m,n}\}$ is said to be definitely divergent.

A double sequence $\textbf{x}=\{x_{m,n}\}$ is bounded if there is an $M>0$  such that $|x_{m,n} |< M$ for all $m, n \in{\textbf{N}}$.
Notice that a $P$-convergent double sequence need not be bounded.

\begin{definition}
[Patterson, 2000 \cite{rfp}] A double sequence $\bf {y}$ is a {\bf double subsequence} of $\bf x$
provided that there exist increasing index sequences $\{n_{j}\}$
and $\{k_{j}\}$ such that, if $\{x_{j}\} = \{x_{n_{j},k_{j}}\}$, then $\bf y$ is
formed by
\[ \begin{array}{cccc}
  x_{1}&x_{2}& x_{5} &x_{10}\\
   x_{4}&x_{3}&x_{6}&-\\
 x_{9}&x_{8}&x_{7}&- \\
 - &-&-&-. \\
        \end{array}\]
\end{definition}

\section{\bf Results}

\begin{definition} (\cite{Boos}, \cite{MoriczTauberiantheoremsforCesarosummabledoublesequences}) \label{Defslowlyoscillatingdoublesequence} A double sequence $\textbf{x}=\{x_{k,l}\}$ of points in $\textbf{R}$ is called slowly oscillating
for any given $\varepsilon>0$, there exist $\alpha=\alpha(\varepsilon)$, $\delta=\delta (\varepsilon) >0$, and $N=N(\varepsilon)$ such that $|x_{k,l}-x_{s,t}|<\varepsilon$, if $k,l\geq N(\varepsilon)$ and $k\leq s \leq (1+\alpha)k$, $l\leq t \leq (1+\delta)l$.
\end{definition}
Any Cauchy double sequence is slowly oscillating, so any P-convergent double sequence is. The converse is easily seen to be false as in the single dimensional case as the following example shows.

\begin{example} Double sequence defined by
let $s_{n}=\log n$
\[ \begin{array}{ccccc}
s_{1}& s_{2}& s_{3}&s_{4}&\cdots\\
 s_{2}& s_{2}& s_{3}&s_{4}&\cdots\\
 s_{3}& s_{3}& s_{3}&s_{4}&\cdots\\
 s_{4}& s_{4}& s_{4}&s_{4}&\cdots\\
 \vdots& \vdots& \vdots&\vdots&\ddots\\
         \end{array}\]
is not $P$-convergent nor Cauchy, however it is a slowly oscillating double sequence.
\end{example}

\begin{theorem} \label{TheoUniformcontimpliesslowlyosclcont} If a factorable double function $f$ defined on a double subset $A\times A$ of ${\bf R}^{2}$
is uniformly continuous, then it preserves factorable slowly oscillating  double sequences from $A\times A$.
\end{theorem}
\begin{proof} Suppose that $f$ is uniformly continuous, and let
\[ \begin{array}{ccccccc}
x_{1,1}& x_{1,2} & x_{1,3} &\cdots\\
      x_{2,1}& x_{2,2} & x_{2,3} &\cdots\\
     x_{3,1}& x_{3,2} &x_{3,3}&\cdots\\
    \vdots&\vdots& \vdots&\ddots
         \end{array}\]
be any slowly oscillating factorable double sequence. To prove that $\{f(x_{n,m})\}$ is slowly oscillating, take any $\varepsilon > 0$. Uniform continuity of $f$ implies that there exists a $\delta >0$ such that $|f(x)-f(y)|< \varepsilon$ whenever  $|x-y|< \delta$ for $x,y\in{A\times A}$ where the absolute value in the latter in $\textbf{R}^{2}$. Since $\{x_{n,m}\}$ is slowly oscillating for this $\delta$, there exist $\alpha_{1}=\alpha_{1}(\delta) >0$, $\delta_{1}=\delta_{1} (\delta)>0$ and $N=N(\delta)$ such that $|x_{k,l}-x_{s,t}|<\delta$, if $k,l\geq N(\delta)$ and $k\leq s \leq (1+\alpha_{1})k$, $l\leq t \leq (1+\delta_{1})l$. Hence $|x_{k,l}-x_{s,t}|<\delta$, if $k,l\geq N(\delta )$ and $k\leq s \leq (1+\alpha_{1})k$, $l\leq t \leq (1+\delta_{1})l$. It follows from this that
$\{f(x_{n,m})\}$ is slowly oscillating. This completes the proof of the theorem.
\end{proof}

\begin{theorem} \label{TheoFunctionspreservesdoublequasiCauchysequencesarecontinuous}
If a factorable double function $f$ defined on a double subset $A\times A$ of ${\bf R}^{2}$ preserves factorable slowly oscillating double sequences from $A\times A$, then it preserves factorable  P-convergent double sequences from $A\times A$.

\end{theorem}

\begin{proof} Suppose that $f$ preserves factorable slowly oscillating double sequences from $A\times A$. Let
\[ \begin{array}{ccccccc}
a_{1,1}& a_{1,2} & a_{1,3} &\cdots\\
      a_{2,1}& a_{2,2} & a_{2,3} &\cdots\\
     a_{3,1}& a_{3,2} &a_{3,3}&\cdots\\
    \vdots&\vdots& \vdots&\ddots
         \end{array}\]
be any $P$-convergent factorable double sequence with $P$-limit $L$. Then the sequence \[ \begin{array}{ccccccc}
a_{1,1}\;L\;& a_{1,2}\;L\;& a_{1,3}\;L\; &...\\
L&L& L &L& L &L&...\\
     a_{2,1}\;L\;& a_{2,2}\;L\;& a_{2,3}\;L\; &\cdots\\
   L&L& L &L& L &L&...\\
   a_{3,1}\;L\;& a_{3,2}\;L\;& a_{3,3}\;L\;&\cdots\\
  L&L& L &L& L &L&...\\
    \vdots&\vdots& \vdots&\vdots&\vdots& \vdots&\ddots
         \end{array}\]
is also $P$-convergent with $P$-limit $L$. Since any P-convergent double sequence is slowly oscillating, this sequence is slowly oscillating. So the transformed sequence of the sequence is slowly oscillating. Thus it follows that

\[ \begin{array}{ccccccc}
f(a_{1,1})\;f(L)\;& f(a_{1,2})\;f(L)\;& f(a_{1,3})\;f(L)\; &...\\
f(L)&f(L)& f(L) &f(L)& f(L) &f(L)&...\\
     f(a_{2,1})\;f(L)\;& f(a_{2,2})\;f(L)\;& f(a_{2,3})\;f(L)\; &\cdots\\
   f(L)&f(L)& f(L) &f(L)& f(L) &f(L)&...\\
   f(a_{3,1})\;f(L)\;& f(a_{3,2})\;f(L)\;& f(a_{3,3})\;f(L)\;&\cdots\\
  f(L)&f(L)&f(L) &L& f(L) &f(L)&...\\
    \vdots&\vdots& \vdots&\vdots& \vdots&\vdots&\ddots
         \end{array}\]
is a slowly oscillating double sequence. Hence \[ \begin{array}{ccccccc}
f(a_{1,1})-f(L)\;& f(a_{1,2})-f(L)\;& f(a_{1,3})-f(L)\; &...\\
     f(a_{2,1})-f(L)\;& f(a_{2,2})-f(L)\;& f(a_{2,3})-f(L)\; &\cdots\\
     f(a_{3,1})-f(L)\;& f(a_{3,2})-f(L)\;& f(a_{3,3})-f(L)\;&\cdots\\
    \vdots&\vdots& \vdots&\ddots
         \end{array}\]
a $P$-convergent factorable double sequence with $P$-limit $0$. This implies that the transformed double sequence
\[ \begin{array}{ccccccc}
f(a_{1,1})\;& f(a_{1,2})\;& f(a_{1,3})\;&...\\
     f(a_{2,1})\;& f(a_{2,2})\;& f(a_{2,3})\; &\cdots\\
     f(a_{3,1})\;& f(a_{3,2})\;& f(a_{3,3})\;&\cdots\\
    \vdots&\vdots& \vdots&\ddots
         \end{array}\]
is $P$-convergent with $P$-limit $f(L)$. This completes the proof of the theorem.

\end{proof}

\begin{corollary} If a factorable double function $f$ defined on a double subset $A\times A$ of ${\bf R}^{2}$ preserves factorable slowly oscillating double sequences from $A\times A$, then it preserves $\lambda$-statistically convergent (single) sequences from $A\times A$.

\end{corollary}

\begin{proof}
The proof follows from the regularity and subsequentiality of $\lambda$-statistically sequential method so is omitted (\cite{CakalliandSonmezandAraslamdastatisticallywardcontinuity}).

\end{proof}

\begin{theorem} \label{TheouniformcontinuityimpliespreservingquasidoublequasiCauchyness}
Suppose that $A\times A$ is a bounded subset of $R^2$. A two dimensional factorable real-valued function is uniformly continuous on $A\times A$ if and only if it preserves factorable slowly oscillating double sequences from $A\times A$.
\end{theorem}
\begin{proof}
It immediately follows from Theorem \ref{TheoUniformcontimpliesslowlyosclcont} that two dimensional uniformly continuous functions preserve slowly oscillating sequences. Conversely, suppose that $f$ defined on $A \times A$ is not uniformly continuous. Then there exists an $\epsilon >0$ such that for any $\delta>0$ there exist  $(a,b),(\bar{a},\bar{b})\in A\times A$ with $\sqrt{(a-\bar{a})^{2} + (b-\bar{b})^{2}}<\delta$
but  $\abs{f(a,b)-f(\bar{a},b)}\geq \epsilon$,  $\abs{f(a,b)-f(a,\bar{b})}\geq \epsilon$, and  $\abs{f(a,b)-f(\bar{a},\bar{b})}\geq \epsilon$, respectively. Thus   for each positive integer $n$ we can choose $(a_{n},b_{n}),(\bar{a}_{n},\bar{b}_{n})\in A\times A$ with $\sqrt{(a_{n}-\bar{a}_{n})^{2} + (b_{n}-\bar{b}_{n})^{2}}<\frac{1}{n}$
but  $\abs{f(a_{n},b_{n})-f(\bar{a}_{n},b)}\geq \epsilon$,  $\abs{f(a_{n},b_{n})-f(a_{n},\bar{b}_{n})}\geq \epsilon$, and  $\abs{f(a_{n},b_{n})-f(\bar{a}_{n},\bar{b}_{n})}\geq \epsilon$,
Then since $A\times A$ is bounded there exists a slowly oscillating double subsequence of the double sequence $\{a_{n},b_{n}\}$ by a simple extension of Bolzano-Weierstrass theorem, $\{a_{n_{k}}, b_{n_{k}}\}$ say. Thus the corresponding double sequence $\{\bar{a}_{n_{k}},\bar{b}_{n_{k}}\}$ has a slowly oscillating double subsequence, say $\{\bar{a}_{n_{k_{m}}},\bar{b}_{n_{k_{m}}}\}$. It is easy to see that $\{\bar{a}_{n_{k_{m}}}, \bar{b}_{n_{k_{m}}}\}$ is a slowly oscillating sequence.
Since $f$ preserves slowly oscillating double sequences by the hypothesis, $\{f(a_{n_{k_{m}}}, b_{n_{k_{m}}})\}$ and $\{f(\bar{a}_{n_{k_{m}}},\bar{b}_{n_{k_{m}}})\}$ are slowly oscillating. This is impossible.
This contradiction completes the proof of the theorem.
\end{proof}

It is well known that uniform limit of a sequence of continuous functions is continuous. This is also true for two dimensional factorable real-valued functions that preserve slowly oscillating double sequences, i.e. uniform limit of a sequence of two dimensional factorable real-valued functions preserving slowly oscillating double sequences from  $A \times A$ of $\textbf{R}^2$ also preserves slowly oscillating double sequences from $A \times A$.

\begin{theorem}
If $(f_{n})$ is a sequence of two dimensional factorable real-valued functions preserving slowly oscillating double sequences from a
double interval $I \times I$ of $\textbf{R}^2$ and $(f_{n})$ is uniformly convergent to a function $f$, then $f$ preserves slowly oscillating double sequences from $I \times I$.
\end{theorem}
\begin{proof}
Let $(x_{nk})$ be a slowly oscillating double sequence  and $\varepsilon > 0$. Then there exists a positive integer $N$ such that $|f_{n}(a,b)-f(\bar{a},\bar{b})|<\frac{\varepsilon}{3}$ for all $(a,b),(\bar{a},\bar{b})\in I\times I$ whenever $n\geq N$. As $f_{N}$ preserves slowly oscillating double sequences from $I \times I$, there exist a $\delta >0$ and a positive integer $N_{1}=N_{1}(\varepsilon)$, greater than $N$, such that $$|f_{N}(x_{k,l})-f_{N}(x_{s,t})|<\frac{\varepsilon}{3}$$ for $n\geq N_{1}$ and $k\leq s \leq (1+\delta)k$, $l\leq t \leq (1+\delta)l$.
 Now for  $n\geq N_{1}$ and $k\leq s \leq (1+\delta)k$, $l\leq t \leq (1+\delta)l$. Thus for  $n\geq N_{1}$ and $k\leq s \leq (1+\delta)k$, $l\leq t \leq (1+\delta)l$ we have $$ |f(x_{k,l})-f(x_{s,t})|\leq |f(x_{k,l})-f_{N}(x_{k,l})|+|f_{N}(x_{k,l})-f_{N}(x_{s,t})|+|f_{N}(x_{s,t})-f(x_{s,t})| $$ $$\leq \frac{\varepsilon}{3} + \frac{\varepsilon}{3} + \frac{\varepsilon}{3}= \varepsilon.$$
This completes the proof of the theorem.
\end{proof}

\begin{theorem}
If $(f_{m,n})$ is a double sequence of two dimensional factorable real-valued functions preserving slowly oscillating double sequences from a
double interval $I \times I$ of $\textbf{R}^2$ and $(f_{m,n})$ is uniformly P-convergent to a function $f$, then $f$ preserves slowly oscillating double sequences from $I \times I$.
\end{theorem}
The proof is similar to the last theorem and as of such it is omitted.
\section{Conclusion}
It is easy to see that double Cauchy sequences are slowly oscillating double. The converse is easily seen to be false as in the single dimensional case (\cite{CakalliSlowlyoscillatingcontinuity}, \cite{Vallin}, \cite{DikandCanak}).
One should also note that are nice connections between double slowly oscillating sequences and uniform continuity of two-dimensional real-valued functions. This is illustrated through the following theorem.
Suppose that $I\times I$ is any two dimensional bounded interval. Then a two dimensional factorable real-valued function is uniformly continuous on $I\times I$ if and only if it is defined on $I\times I$ and preserves factorable double slowly oscillating sequences from $I\times I$. Extensions and variations of the above theorem was also presented.

\bibliographystyle{amsplain}

\end{document}